\documentclass[12pt,reqno]{article}

\usepackage{color}
\definecolor{webgreen}{rgb}{0,.5,0}
\definecolor{webbrown}{rgb}{.6,0,0}
\definecolor{redmaple}{rgb}{.73,.22,.05}
\definecolor{bluemaple}{rgb}{.29,0,.71}
\definecolor{extraproof}{rgb}{0,.5,.5}

\usepackage{amssymb}
\usepackage{amsmath}
\usepackage{amsthm}
\usepackage{amsfonts}
\usepackage{graphicx}
\usepackage{cancel}
\usepackage[colorlinks=true,
linkcolor=webgreen,
urlcolor=bluemaple,
filecolor=webbrown,
citecolor=webgreen]{hyperref}

\usepackage[symbol]{footmisc}

\usepackage{multirow}

\usepackage{enumitem,kantlipsum}

\usepackage{subfigure}

\begin{document}

\newtheorem{theorem}{Theorem}
\newtheorem*{theorem*}{Theorem}
\newtheorem{lemma}{Lemma}
\newtheorem{proposition}{Proposition}
\newtheorem*{proposition*}{Proposition}
\newtheorem{corollary}{Corollary}
\newtheorem*{corollary*}{Corollary}
\newtheorem{conjecture}{Conjecture}
\newtheorem{remark}{Remark}
\newtheorem{problem}{Problem}

\theoremstyle{definition}
\newtheorem{definition}{Definition}
\newtheorem*{definition*}{Definition}
\newtheorem{example}{Example}
\newtheorem*{example*}{Example}
\newtheorem{notation}{Notation}

\newcommand{\R}{{\mathbb R}}
\newcommand{\Q}{{\mathbb Q}}
\newcommand{\C}{{\mathbb C}}
\newcommand{\N}{{\mathbb N}}
\newcommand{\Z}{{\mathbb Z}}

\newcommand{\seqnum}[1]{\href{https://oeis.org/#1}{#1}}

\newcommand{\doi}[1]{\href{http://doi.org/#1}{DOI: #1}}

\newcommand{\arxiv}[1]{\href{https://arxiv.org/#1}{arXiv: #1}}

\makeatletter

\renewcommand\@makefntext[1]{%
\setlength\parindent{1em}%
\noindent
\mbox{$^\@thefnmark$~}{#1}}

\makeatother

\begin{center}
\vskip 1cm{\LARGE\bf 
An extension of Furstenberg's theorem

\vskip 0,3cm
of the infinitude of primes\footnote[2]{JP J. Algebra Number Theory Appl. 53(1) (2022), 21-43. \doi{10.17654/0972555522002}}}
\vskip 0,6cm
\large
F. Javier de Vega\\
King Juan Carlos University\\
Madrid, Spain\\
\href{mailto:javier.devega@urjc.es}{\tt javier.devega@urjc.es} \\
\end{center}

\vskip .2 in

\begin{abstract}
The usual product $m\cdot n$ on $\Z$ can be viewed as the sum of $n$ terms of an arithmetic progression whose first term is $a_{1}=m-n+1$ and whose difference is $d=2$. Generalizing this idea, we define new similar product mappings, and we consider new arithmetics that enable us to extend Furstenberg's theorem of the infinitude of primes. We also review the classic conjectures in the new arithmetics. Finally, we make important extensions of the main idea. We see that given any integer sequence, the approach generates an arithmetic on integers and a similar formula to $\Z \setminus \{-1,1\}$ arises.
\end{abstract}



%%%%%%%%%%%%%%%%%%%%%%%%%%%%%%%%%%%%%%%%%%%%%%%%%%%%%%%%%%%%%%%%%%%%%%%%%%%%%%%%%%%%%%%%%%%%%%%%%%%%%%%%%%%%%%%%%%%%%%%%%%%%%%%%%%%%%%%%%%%%

\section{Introduction}\label{sec:intro}
In 1955, Furstenberg \cite{furstenberg} proposed a topological proof of the infinitude of primes. He considered the arithmetic progression topology on the integers, where a basis of open sets is given by subsets of the following form: for $a, b \in \Z$, $a>0$,  $S(a, b) = \{ n \cdot a  + b: n \in \Z \}$. Arithmetic progressions themselves are by definition open. He studied this family of open sets and reached the following expression:
\begin{equation}
\label{eq:1}
\bigcup_{p \mathrm{\, prime }} S(p, 0)=\Z \setminus \{-1, \ 1 \}.
\end{equation}

\noindent With \eqref{eq:1}, he concluded that the set of primes is infinite; see \cite{furstenberg} for more details.

Other extensions of this result propose some observations on the Furstenberg topological space (see \cite{golomb1,golomb2,lovas}) or present new variants of Furstenberg topological proof based on tools taken from commutative algebra (see \cite{knopfmacher,porubsky}). 

Our purpose will be different from the previous ones. We will see that in a way, every integer sequence $(a_{n})_{n>0}$ generates a formula similar to \eqref{eq:1}. The main idea is as follows:

The usual product $m\cdot n$ on $\Z$ can be viewed as the sum of $n$ terms of an arithmetic progression $(a_n)$ whose first term is $a_{1}=m-n+1$ and whose difference is $d=2$.
\begin{example*} $5\cdot 3 = (5-3+1)+5+7=3\cdot 5=(3-5+1)+1+3+5+7$.
\end{example*}

It is thus natural to consider the following questions:

\begin{enumerate}
	\item[I.] \label{introduction:it:1} What if one were to consider product mappings, similar to the usual case, by varying the first term $a_{1}$?
	\item[II.] \label{introduction:it:2} What if one were to consider product mappings by varying the distance $d=k$?
\end{enumerate}

Question II is more interesting than the first. By studying this question, we obtain the definition of a family of product mappings on integers.

\begin{definition*}
Given $m,n,k\in\Z$, $n>0$, we define  
\begin{equation*}
m\odot_{_{k}} n =(m-n+1)+(m-n+1+k)+\ldots+(m-n+1+k+\stackrel{(n-1)}{\ldots}+k)
\end{equation*}
as the $k$-\textit{arithmetic} product.
\end{definition*}

In connection with the above result, for each $k \in \Z$, the expression ``\textit{given a} $k$-\textit{arithmetic}'' refers to the fact that we are going to work with integers, the sum, the new product and the usual order. This means that we are going to work on $\displaystyle \mathcal{Z}_k=\{ \Z , + ,\odot_{_{k}} , < \}$. Clearly, $\mathcal{Z}_2$, the $2$-\textit{arithmetic}, will be the usual arithmetic.

Characterizing the concept of divisor and prime number in the new arithmetics, we obtain the following theorem:

\begin{theorem*}
Given a $k$-\textit{arithmetic}, the primes (\textit{arith} $k$) are:
\begin{itemize}
\item \textit{The usual primes if} $k \in E= \{ \ldots ,-4,-2,0,2,4,6,\ldots \}$.
\item \textit{The usual powers of two if} $k \in O= \{ \ldots ,-3,-1,1,3,\ldots \}$.
\end{itemize}
\end{theorem*}

Now, using the previous result, we can apply Furstenberg's proof on $\mathcal{Z}_{k}$. 

\begin{theorem*}
For all integer $k$, there are infinitely many primes on $\mathcal{Z}_{k}$.
\end{theorem*}

This theorem is a trivial consequence of the previous one, but the novelty of the proof that we offer is the application of Furstenberg's method on $\mathcal{Z}_{k}$. Roughly speaking, we define $S_{k}(a, b) = \{ n\odot_{_{k}} a  + b: n \in \Z \}$ and obtain:
\begin{equation} \label{eq:2}
\begin{split}
&k \in E \Rightarrow \bigcup_{p \mathrm{\, prime \ }(arith \ k)} S_{k}(p, 0)=\Z \setminus \{-1, \ 1 \}. \\
&k \in O \Rightarrow \bigcup_{p \mathrm{\, prime \ }(arith \ k)} S_{k}(p, 0)=\Z \setminus \{0 \}.
\end{split} \end{equation}
\noindent With \eqref{eq:2}, we will prove the result. The above formula relates the usual primes and the powers of two.

Another interesting point is to study the classical conjectures in the new arithmetics. One result in this line is the following.

\begin{theorem*}
If Goldbach's conjecture is true in the usual arithmetic, then the conjecture must be true in $\mathcal{Z}_k, \ k \in E$.
\end{theorem*}

In other words, if $k$ is even, the conjecture is equivalent in all $k$-\textit{arithmetics}. This result is notable because the $k$-\textit{arithmetic} product ($\odot_{_{k}}$) is not commutative and not associative if $k \neq 2$. We will also revise Collatz conjecture.

Finally, we make important extensions of the main idea. We see that given any integer sequence, the approach will generate an arithmetic on integers. We will evaluate, via an algebraic computation system, some examples of arithmetics generated by sequences and attempt to study the primes, the squares and an analogue of formula \eqref{eq:2} in the simplest cases. In this part of the paper, new sequences and connections appear.

We have briefly presented the idea of this paper. We now develop the new arithmetics and their connections with Furstenberg's theorem and the classic conjectures in the usual arithmetic.

%%%%%%%%%%%%%%%%%%%%%%%%%%%%%%
\section{Basic definitions and properties} \label{sec:2}
We now start to construct the definition of $\mathcal{Z}_k$, that is, the set of integers with the sum, usual order and a product mapping similar to the usual one.
\begin{definition}[$k$-\textit{arithmetic product} $\odot_{k}$]
\label{def:1}
Given $m,k\in \Z$, for all positive integers $n$, we define the following expression 
\begin{equation*}
m\odot_{_{k}} n =(m-n+1)+(m-n+1+k)+\ldots+(m-n+1+k+\stackrel{(n-1)}{\ldots}+k)
\end{equation*}
as the $k$-\textit{arithmetic} product.
\end{definition}

This arithmetic progression can be added to obtain the following formula:
\begin{equation}
\label{eq:3}
m \odot_{_{k}} n =(m-n+1)\cdot n + \frac{n\cdot (n-1)\cdot k}{2}.
\end{equation}

We take \eqref{eq:3} as Definition \ref{def:1} and consider $n \in \Z$. Observe that the usual product is used to define the $k$-\textit{arithmetic product} (which includes the usual one if $k=2$).

Our product can be formalized with the successor operation in Peano arithmetic.
Recall the definition of the product between two natural numbers $m$, $n$ using the successor operation $S(n)$: $P(m,n)$ is a number such that
\begin{equation} \label{eq:4}
\begin{split}
&\bullet P(m,1)=m. \\
&\bullet P(m,S(n))=m+P(m,n).
\end{split}
\end{equation}                  

Similarly to \eqref{eq:4}, we can consider  $\mathcal{N} = \{1, S, +, < \}$ (Peano arithmetic with only the successor operation, the sum and the usual order) and define the following product operation:

\begin{definition}[$t$-\textit{Peano product}]
\label{def:2}
Given $t \in \mathcal{N}$, to every pair of numbers $m,n \in \mathcal{N}$, we may assign in exactly one way a natural number, called $P_t(m,n)$, such that
\begin{equation*}
\begin{split}
&\bullet P_t(m,1)=m. \\
&\bullet P_t(m,S(n))=m+P_t(m+t,n).
\end{split}
\end{equation*}
$P_t(m,n)$ is called the $t$-\textit{Peano product} of $m$ and $n$.
\end{definition}

Now, we can relate the $k$-\textit{arithmetic} product and the $t$-\textit{Peano product}.
\begin{proposition} Given $m,n,t\in\mathcal{N}$, then $P_{t}(m,n)=m\odot_{_{t+2}} n$. 
\label{propo:1} 
\end{proposition}
\begin{proof} We need only Definitions \ref{def:2} and \ref{def:1}.
\begin{equation*}
\begin{split}
P_t(m,n) & = P_t(m,S(n-1))=m+P_t(m+t,n-1)\\
& = m+P_t(m+t,S(n-2))=m+m+t+P_t(m+t+t,n-2)\\
& =\ldots= m+(m+t)+(m+t+t)+\ldots+(m+t+\stackrel{(n-1)}{\ldots}+t)\\
& = mn+\frac{n(n-1)t}{2}.
\end{split}
\end{equation*}
\noindent On the other hand:
\begin{equation*}
m \odot_{_{t+2}}n=(m-n+1)n+\frac{n(n-1)(t+2)}{2}=mn+\frac{n(n-1)t}{2}.
\end{equation*}
\noindent The result follows.
\end{proof}

We have studied this proposition to show that the new products are similar to the usual one and could have arisen from Peano arithmetic. Later, this result will motivate us to make interesting conjectures and heuristic reasoning. 
However, the following theorem provides a better understanding of the $k$-\textit{arithmetic} product.

\begin{theorem}[$k$-\textit{arithmetic polygonal theorem}]
\label{theorem:1}
The product $n\odot_{_{k}} n$ is the $n$th $(k+2)$-gonal number.
\end{theorem}
\begin{proof}
The formula of the $n$th $l$-gonal is $\frac{1}{2}n((l-2)n-(l-4))$. We have, by hypothesis, $k+2=l$; hence,

\[
\pushQED{\qed}
\frac{n((l-2)n-(l-4))}{2}=\frac{n(kn-(k-2))}{2}=\frac{nk(n-1)}{2}+n = n\odot_{_{k}}n. \qedhere
\popQED
\]
\renewcommand{\qedsymbol}{} \vspace{-\baselineskip}
\end{proof}

This theorem is easy to prove but has great significance. The $k$-\textit{arithmetic} product generalizes the usual product in the following way: the \textit{square of a number} is a square (polygon) in the usual arithmetic but  a triangle in $1$-\textit{arithmetic}, a pentagon in $3$-\textit{arithmetic}, a hexagon in $4$-\textit{arithmetic}, etc. That is, $4\odot_{_{2}} 4=16$ is the fourth square, whereas $4\odot_{_{1}} 4=10$ is the fourth triangular number and $4\odot_{_{3}} 4=22$ is the fourth pentagonal number. If we want to calculate the 5th pentagonal number, we can do $5\odot_{_{3}} 5=35$.

From now on, we will work on $\mathcal{Z}_k=\{ \Z , + ,\odot_{_{k}}, < \}$. If $k\neq 2$, the $k$-\textit{arithmetic} product $\odot_{_{k}}$ is not commutative and not associative; thus, the group or ring structures are not considered in this paper. However, we have elementary algebraic properties that connect the usual product with the others. For instance:

\begin{proposition} 
\label{propo:2}
Given $a, b, c, d, k \in \Z$, the following properties are satisfied:
\begin{enumerate}
\item The $k$-\textit{arithmetic} product is not associative in general.\\
If $k\neq 2$ and $c\neq \{ 0, 1 \}$, then $(a\odot_{_{k}} b)\odot_{_{k}} c \neq a \odot_{_{k}}(b\odot_{_{k}} c)$. \label{proposition:2:1}
\item The $k$-\textit{arithmetic} product is not commutative in general but \\
$a\odot_{_{k}} (1-a)=(1-a)\odot_{_{k}} a$. \label{pro:2:2}
\item $(a-b)\cdot(c+d)=a\odot_{_{k}}c+a\odot_{_{k}}d-b\odot_{_{k}}c-b\odot_{_{k}}d$. \label{proposition:2:3}
\item $(a+b)\odot_{_{k}}(a+b)=a\odot_{_{k}}a+b\odot_{_{k}}b+k\cdot a \cdot b$. \label{proposition:2:4}
\item $(a-b)^2=a\odot_{_{k}}a+b\odot_{_{k}}b-b\odot_{_{k}}a-a\odot_{_{k}}b$. \label{proposition:2:5}
\item $a \odot_{k} (-b)=(k-2-a) \odot_{k} b$. \label{proposition:2:6}
\end{enumerate}
\end{proposition}
\begin{proof}
Only we have to use formula \eqref{eq:3}. 
\end{proof}

It is not the purpose of this paper to study these types of properties.

The author does not known references where these alternative varieties of integer multiplication appear. Applegate et al. \cite{applegate1,applegate2} worked on original arithmetics but there are not connections. The $k$-\textit{arithmetic} product symbol $\odot_{k}$ is also use in ``tropical mathematics'' \cite{richter-gebert}, where addition and multiplication are defined by $x \oplus y = \min\{x,y\}$, $x \odot y = x + y$ but there is no connection either.

In the following chapter, we will extend the definition of divisor and prime number and obtain the fundamental theorem that will allow us to extend Furstenberg's theorem of the infinitude of primes.

%%%%%%%%%%%%%%%%%%%%%%%%%%%%%%
\section{Divisors and primes} \label{sec:3}
\begin{definition}[$k$-\textit{arithmetic divisor}]
\label{def:3}
\noindent Given a $k$-\textit{arithmetic}, an integer $d>0$ is called a divisor of $a$ (\textit{arith} $k$) if there exists some integer $b$ such that $a=b \odot_{_{k}} d$. We write: 
\begin{equation*}
d \mid a \ (arith \ k) \Leftrightarrow \exists b \in \Z \textrm{ such that } b\odot_{k} d=a.
\end{equation*}
\end{definition}

In other words, $d$ is the number of terms of the summation that represents the $k$-\textit{arithmetic} product (see the following example).

\begin{example}
\label{example:1}
Consider the following expression: $\displaystyle 6 \odot_{3} 5=2+5+8+11+14=40$. The number of terms is $5$; hence, we can say that $5$ is a divisor of $40$ in $3$-\textit{arithmetic}, that is, $5$ is a divisor of $40$ (\textit{arith} $3$). Notably, a divisor is always a positive number, and the number $6$ indicates where we should start the summation. However, we cannot be sure that $6$ is a divisor of $40$ (\textit{arith} $3$). Another point is to consider the expression $6 \odot_{3} (-5)$. Is $-5$ a divisor? We can use point \ref{proposition:2:6} of Proposition \ref{propo:2}: $6 \odot_{3} (-5)=(3-2-6)\odot_{3} 5 = -15$. We can say that $5$ is a divisor of $-15$ (\textit{arith} $3$).
\end{example}
To characterize the set of divisors, we define the $k$-\textit{arithmetic} quotient:

\begin{definition}[$k$-\textit{arithmetic quotient} $\oslash_{k}$]
\label{def:4}
\noindent Given a $k$-\textit{arithmetic}, an integer $c$ is called a quotient of $a$ divided by $b$ (\textit{arith} $k$) if and only if $c\odot_{_{k}} b=a$. We write: 
\begin{equation*}
a \oslash_{k} b = c \Leftrightarrow c\odot_{_{k}} b=a.
\end{equation*}
\end{definition}

By means of the following proposition, we can use the usual quotient to study the new one.
\begin{proposition}
\label{propo:3}
Given a $k$-\textit{arithmetic} and $a,b,k \in \Z$, $b \neq 0$, 
\begin{equation*}
a \oslash_{k} b=\frac{a}{b}+(b-1)\cdot(1-\frac{k}{2}).
\end{equation*}
\end{proposition}
\begin{proof}
$a \oslash_{k} b=c \Leftrightarrow (c-b+1)b+\frac{1}{2}b(b-1)k=a$. Solving the previous equation for $c$ leads to the result. 
\end{proof}
We must consider $\oslash_{k}$ in the following manner. If we want to write $a$ as the sum of $b$ terms of an arithmetic progression, then the quotient will give us the place to start the summation (see the following example).

\begin{example} Express $81$ as the sum of $6$ terms of an arithmetic progression whose difference is $3$.

\noindent\textit{Solution.} We can then obtain $81\oslash_{3} 6=\frac{81}{6}+5\cdot(1-\frac{3}{2})=11$. Hence, $81=11 \odot_{3} 6$. The first term is $11-6+1=6$, and the solution is: $6+9+12+15+18+21=81$. Clearly, $6$ is a divisor of $81$ (\textit{arith} $3$).
\end{example}

\begin{corollary}
\label{corollary:1}
Let $b>0$, $b$ is a divisor of $a$ (\textit{arith} $k$) $\Leftrightarrow$ $a \oslash_{k} b$ is an integer.
\end{corollary}
\begin{proof}
$a \oslash_{k} b = c \in \Z \ \ \stackrel{Def. \ref{def:4}}{\Leftrightarrow} \ \ c\odot_{k}b=a \ \ \stackrel{Def. \ref{def:3}}{\Leftrightarrow} \ \ b \mid a \ (arith \ k)$. 
\end{proof}

Consider Example \ref{example:1}: $40=6 \odot_{3} 5$ but $6$ is not a divisor of $40$ (\textit{arith} $3$) because $40 \oslash_{3} 6=\frac{40}{6}+5 \cdot (1-\frac{3}{2})=\frac{25}{6} \notin \Z$.

\medskip
For the upcoming Lemma and the rest of this paper, we use the following notation for even and odd numbers.
\begin{notation} We write the set of even and odd numbers as follows: 
\begin{itemize}
	\item $E= \{ \ldots ,-4,-2,0,2,4,6,\ldots \}$.
	\item $O= \{ \ldots ,-3,-1,1,3,5,7,\ldots \}$.
\end{itemize}
\end{notation}

\begin{lemma}
\label{lemma:1}
Given a $k$-\textit{arithmetic} and $a \in \Z$, the divisors of $a$ (\textit{arith} $k$) are:

\begin{enumerate}
	\item The usual divisors of $a$ if $k \in E$.
	\item The usual divisors of $2a$ except the even usual divisors of $a$ if $k \in O$.
\end{enumerate}
\end{lemma}
\begin{proof}
We use Proposition \ref{propo:3} and Corollary \ref{corollary:1} in the following cases:
\begin{enumerate}
	\item[1.] $k \in E$. $d \mid a \Leftrightarrow d \mid a$ (\textit{arith} $k$).
	\item[2a.] $k \in O$. If $d$ is odd: $d \mid 2a \Leftrightarrow d \mid a$ (\textit{arith} $k$).
	\item[2b.] $k \in O$. If $d$ is even: $d \mid 2a$ and $d \nmid a \Leftrightarrow d \mid a$ (\textit{arith} $k$).
\end{enumerate}

We will prove Part 2b, leaving the rest to the reader.

\medskip
2b. $k \in O$. Suppose $d$ is an even usual divisor of $2a$ but $d \nmid a$:
\begin{equation*}
d \mid 2a \Rightarrow \exists h \in \Z \textrm{ such that } 2a=dh \Rightarrow a/d=h/2.
\end{equation*}

By hypothesis $d \nmid a$, hence $h/2 \notin \Z$ and $h$ is odd. Then,
\begin{equation*}
\begin{split}
a \oslash_{k} d &= a/d+(d-1)(1-k/2)=h/2+(d-1)(1-k/2)\\
&=(h-(d-1)k)/2+d-1 \in \Z.
\end{split}
\end{equation*}
The previous expression is an integer because $(h-(d-1)k)$ is even (difference of odd numbers). Hence, $d$ is a divisor of $a$ (\textit{arith} $k$).

\medskip
$k \in O$. Suppose $d$ is an even number and $d$ is a divisor of $a$ (arith $k$):
\begin{equation*}
\begin{split}
d \mid a \ (arith \ k)&\Leftrightarrow \exists b \in \Z \text{ such that } b \odot_{k}d=a \Leftrightarrow\\
&\Leftrightarrow \left\lbrace
\begin{array}{ll}
(b-d+1)+(d-1)k/2=a/d \notin \Z.\\
2(b-d+1)+(d-1)k=2a/d \in \Z.
\end{array}
\right.
\end{split}
\end{equation*}

\noindent  Hence, $d \nmid a$ and $d \mid 2a$ in the usual sense. 
\end{proof}

The following example is an interesting application of the lemma.
\begin{example} \label{example:3}
Express the number $12$ in all possible ways as a sum of an arithmetic progression whose difference is $3$.

\noindent \textit{Solution.} The divisors of $12$ (\textit{arith} $3$) are the usual divisors of $24$ except the even usual divisors of $12$: $\{1, \ \cancel{2}, \ 3, \ \cancel{4}, \ \cancel{6}, \ 8, \ \cancel{12}, \ 24\}$.

\begin{itemize}
\item $d=1\Rightarrow 12\oslash_{3}1=12 \Rightarrow 12=12\odot_{_{3}} 1 \Rightarrow 12=12$.
\item $d=3\Rightarrow 12\oslash_{3}3=3\Rightarrow 12=3\odot_{_{3}} 3 \Rightarrow$ $12=1+4+7$.
\item $d=8 \Rightarrow 12\oslash_{3}8=-2\Rightarrow 12=-2\odot_{_{3}}8\Rightarrow 12=-9-6-3-0+3+6+9+12$. 
\item $d=24 \Rightarrow 12\oslash_{3}24=-11 \Rightarrow 12=-11\odot_{_{3}} 24\Rightarrow 12=-34-31-28- \ldots +29+32+35$.
\end{itemize}
\end{example}

We can use this example to study the representation of a number as a sum of arithmetic progressions. We obtain easily the results previously studied by other authors \cite{bush,mason}. Also, we can examine the number of nondecreasing arithmetic progressions of positive integers with sum $n$, that is, the partition of a number into arithmetic progressions. See sequence number \seqnum{A049988} in the Online Encyclopedia of Integer Sequences (OEIS) \cite{oeis} and \cite{munagi}. However, we leave this approach for another time.

Let us now consider the primes (\textit{arith} $k$). Following the results above, an integer $a>1$ always has two divisors in any $k$-\textit{arithmetic}:
\begin{itemize}
\item If $k \in E$, $1$ and $a$ are divisors of $a$ (\textit{arith} $k$).
\item If $k \in O$, $1$ and $2a$ are divisors of $a$ (\textit{arith} $k$).
\end{itemize}

Thus, we can write the following definition.

\begin{definition}[$k$-\textit{arithmetic prime}]
\label{def:5}
An integer $p>1$ is called a prime (\textit{arith} $k$), or simply a $k$-\textit{prime}, if it has only two divisors (\textit{arith} $k$). An integer greater than $1$ that is not a prime (\textit{arith} $k$) is termed a composite (\textit{arith} $k$).
\end{definition}

With Lemma \ref{lemma:1}, it is easy to characterize the primes (\textit{arith} $k$).

\begin{theorem}[\textit{Fundamental} $k$-\textit{arithmetic theorem}]
\label{theorem:2}
Given a $k$-\textit{arithmetic}, the primes (\textit{arith} $k$) are:
\begin{enumerate}
\item The usual primes if $k \in E= \{ \ldots ,-4,-2,0,2,4,6,\ldots \}$. \label{th2:it:1}
\item The powers of two if $k \in O= \{ \ldots ,-3,-1,1,3,5,7,\ldots \}$. \label{th2:it:2}
\end{enumerate}
\end{theorem}
\begin{proof}
\ref{th2:it:1}. By Lemma \ref{lemma:1}, if $k \in E$, then $d \mid a \Leftrightarrow d \mid a$ (\textit{arith} $k$):
\begin{equation*}
\begin{split}
p \text{ usual prime } &\Leftrightarrow 1 \mid p \text{ and } p \mid p \Leftrightarrow 1 \mid p \ (\textit{arith } k) \text { and } p \mid p \ (\textit{arith } k) \\
&\Leftrightarrow p \text{ is prime } (\textit{arith }k).
\end{split}
\end{equation*}

\ref{th2:it:2}. $k \in O$. If $a>1$ is not a power of two, then $a=2^{s}\cdot b$, $b$ odd and $s$$\in$$\{ 0,1,2,\ldots \}$. By Lemma \ref{lemma:1}, $b \mid a \ (arith \ k)$. In conclusion, $1, 2a, b$ are divisors of $a$ (\textit{arith} $k$); hence, $a$ is not prime (\textit{arith} $k$).

\medskip
$k \in O$. If $a>1$ is a power of two, then $a=2^{s}$, $s \in \{ 1,2, \ldots \}$. By Lemma \ref{lemma:1}, the divisors of $a$ (\textit{arith} $k$) are the usual divisors of $2a$ except the even usual divisors of $a$. Hence, the divisors of $a$ (\textit{arith} $k$) are: $\{1, \ \cancel{2}, \ \cancel{2^{2}}, \ldots, \ \cancel{2^{s}}, \ 2\cdot 2^{s} \}$. $1$ and $2^{s+1}$ are the unique divisors of $a$ (\textit{arith} $k$); thus, $a$ is $k$-prime. 
\end{proof}

We do not study in this paper the representation of an integer as a product of primes (\textit{arith} $k$). Because we are not in a unique factorization domain, we will have cases like the following: $15=8 \odot_{1} 2 = [(2 \odot_{1} 2)\odot_{1}(2 \odot_{1} 2)]\odot_{1} (2 \odot_{1} 2)$.

Now, we can extend Furstenberg’s theorem of the infinitude of primes.

%%%%%%%%%%%%%%%%%%%%%%%%%%%%%%
\section{The extension of Furstenberg’s theorem}
\label{sec:4}
Attempting to adapt classic arguments about the infinity of primes, we observe an extension of Furstenberg’s theorem. The novelty of this proof is that we are applying Furstenberg's method in arithmetics with noncommutative and nonassociative products. We adapt the version of the original proof \cite{furstenberg} in \cite[p.~5]{aigner}.
\begin{theorem}
\label{theorem:3}
For all integer $k$, there are infinitely many primes on $\mathcal{Z}_{k}$.
\end{theorem}
\begin{proof} For each $k \in \Z$, we are going to define a topology on $\mathcal{Z}_{k}$.

For $a,b \in \Z,a>0$, we set $S_{k}(a, b) = \{ n\odot_{_{k}} a  + b: n \in \Z \}$. Each set $S_{k}(a, b)$  is a infinite arithmetic progression whose difference is $a$ for all $k$. Now call a set $U \subseteq \Z$ open if either $U$ is empty, or if to every $h \in U$, there exists some $a,b \in \Z,a>0$ with $h \in S_{k}(a, b) \subseteq U$. As in the original proof, for each $k \in \Z$, we have a topology on $\mathcal{Z}_{k}$. Let us note two facts: 
\begin{enumerate}
\item[(A)] Any nonempty open set is infinite. \label{pointA}
\item[(B)] Any set $S_{k}(a, b)$ is closed. \label{pointB}
\end{enumerate}
Point (A) is clear. For (B), we observe that $\displaystyle S_{k}(a, b)=\Z \setminus \bigcup_{\substack{i=1}}^{a-1} S_{k}(a, b+i)$. Hence, $S_{k}(a, b)$ is the complement of an open set.

Now, we consider $\bigcup_{p} S_{k}(p, 0)$, where $p$ is prime (\textit{arith} $k$). There are two possibilities:
\begin{equation} \begin{split}
\label{eq:5}
&k \in E \Rightarrow \bigcup_{p \mathrm{\, prime \ }(arith \ k)} S_{k}(p, 0)=\Z \setminus \{-1, \ 1 \}. \\
&k \in O \Rightarrow \bigcup_{p \mathrm{\, prime \ }(arith \ k)} S_{k}(p, 0)=\Z \setminus \{0 \}.
\end{split} \end{equation}
The first possibility is easy to check. If $k \in E$, primes (\textit{arith} $k$) are the usual ones. Moreover, $S_{k}(p,0)=S_{2}(p,0)$. Since any number $h \neq 1,-1$ has a prime divisor $p$ and hence is contained in $S_{k}(p,0)$, the first possibility is proved.

If $k \in O$, primes (\textit{arith} $k$) are the powers of two. If $h=\pm \ 2^{s}\cdot m$, $m$ odd, $s \in \{0,1,\ldots \}$, then $h \in S_{k}(2^{s+1}, 0)$.

Additionally, if we suppose that $0 \in S_{k}(2^{t}, 0)$ for some $t \in \{ 1, 2, \ldots \}$, then there must exist $c \in \Z$ such that $c \odot_{k} 2^{t}=0$. However, $\displaystyle c \odot_{k} 2^{t}=0 \Leftrightarrow (2^{t}-1)(k-2)+2c=0$. This contradiction proves the second possibility because $(2^{t}-1)(k-2)$ is odd and $2c$ is even.

Finally, if primes (\textit{arith} $k$) were finite, then $\bigcup_{p} S_{k}(p, 0)$ would be a finite union of closed sets (by (B)), and hence closed. Consequently, $ \{-1, \ 1 \}$ and \{0\} would be open sets, in violation of (A). 
\end{proof}

Interestingly, the extension of the theorem consists of proving that there are infinitely many powers of two. Regardless, formula \eqref{eq:5} completes or extends Furstenberg's theorem of the infinitude of primes. As we will see later, we will adopt the previous argument when we study new arithmetics.

%%%%%%%%%%%%%%%%%%%%%%%%%%%%%
\section{Classic problems revisited}
\label{sec:5}
We now study some classic conjectures of number theory. We start with Goldbach's Conjecture and observe an interesting property. If $k \in E$, the conjecture is equivalent in all $k$-\textit{arithmetics}. If $k \in O$, the conjecture is false. The same result occurs with more important conjectures, which reminds us of what happens in geometry with the postulate of parallels: the concept of ``line between two points'' depends on the space we are considering. Understanding this idea, the euclidean parallel postulate is solved with simplicity (it can be true or false).

\begin{definition}[$k$-\textit{Goldbach property}]
Let $H=\{6,8,10,\ldots\}$. We say that $\mathcal{Z}_{k}$ has the $k$-\textit{Goldbach} property, denoted by $\mathcal{Z}_k \vDash G_k$, if for all $h \in H$, there exist $p_{1}$,  $p_{2}$ primes (\textit{arith} $k$) such that $p_1+p_2 = h$.
\end{definition}

The usual Goldbach conjecture could be translated: $\mathcal{Z}_{2} \vDash G_{2}$. Now, we have the following theorem.

\begin{theorem}[\textit{Relation with Goldbach's conjecture}]
\label{theorem:4}
If Goldbach's conjecture is true in the usual arithmetic, then the conjecture must be true in $\mathcal{Z}_k, \ k \in E$. Additionally, Goldbach's conjecture is false in $\mathcal{Z}_k, \ k \in O$. That is:
\begin{itemize}
\item If $k \in O$, then $\mathcal{Z}_k \vDash \neg G_k$.
\item If $k \in E$, then $\mathcal{Z}_2 \vDash G_{2} \Leftrightarrow \mathcal{Z}_k \vDash G_k$.
\end{itemize}
\end{theorem}
\begin{proof}
The result is an obvious consequence of Theorem \ref{theorem:2}. If $k \in O$, primes (\textit{arith} $k$) are the powers of two, the conjecture is clearly false. For instance, $14$ in not the sum of two powers of two. If $k \in E$, primes (\textit{arith} $k$) are the usual primes, and the conjecture is equivalent in $\mathcal{Z}_k, \ k \in E$.
\end{proof}

This is an important fact because if $k\neq 2$, then $\odot_{_{k}}$ is not commutative and not associative. 

This type of argument can be made similarly for the twin prime conjecture, Sophie Germain conjecture and Euclid primes conjecture. We leave the formalization of these problems for another time and focus on the Collatz conjecture.
\begin{problem}[\textit{Relation with Collatz conjecture}]
\noindent The \textit{k-Collatz problem} is similar to the usual one but uses $\odot_{k}$ and $\oslash_{k}$.
\begin{equation*}
f_{k}(n) = \begin{cases} n\oslash_{k}2 & \mbox{if }2\mbox{ is a divisor of }n \ (arith \ k), \\ n\odot_{_{k}} 3+1 & \mbox{if }2\mbox{ is not a divisor of }n \ (arith \ k). \end{cases}
\end{equation*}

\noindent We consider the orbit of an integer $n$: $f_{k}(n) \rightarrow f_{k}(f_{k}(n)) \rightarrow \ldots$.
\end{problem}

\begin{example}
\label{example:4}
In this example, we consider the $17$-orbit varying $k$.
\begin{itemize}
\item $k=2$ (The usual Collatz conjecture): $17 \rightarrow 52 \rightarrow 26 \rightarrow 13 \rightarrow 40 \rightarrow 20 \rightarrow 10 \rightarrow 5 \rightarrow 16 \rightarrow 8 \rightarrow \mathbf{4} \rightarrow 2 \rightarrow 1 \rightarrow \mathbf{4} \ \ldots$
\item $k=6$ (There is a cycle of length $8$): $17 \rightarrow 64 \rightarrow 30 \rightarrow \stackrel{(10 \ steps)}{\ldots} \rightarrow \mathbf{34} \rightarrow 15 \rightarrow 58 \rightarrow 27 \rightarrow 94 \rightarrow 45 \rightarrow 148 \rightarrow 72 \rightarrow \mathbf{34}  \ \ldots$
\item $k=1700$ (There is a cycle of length $1124$): $17 \rightarrow 5146 \rightarrow 1724 \rightarrow  13 \rightarrow 513\rightarrow 1718 \rightarrow \stackrel{(24 \ steps)}{\ldots} \rightarrow \mathbf{3730} \rightarrow \stackrel{(1123 \ steps)}{\ldots} \rightarrow \mathbf{3730}  \ \ldots$
\item $k=1$ (The orbit diverges): $17 \rightarrow 9 \rightarrow 5 \rightarrow 3 \rightarrow 2 \rightarrow 4 \rightarrow 10 \rightarrow 28 \rightarrow 82 \rightarrow 244 \rightarrow 730 \rightarrow 2188 \ \ldots$
\item $k=5$ (The orbit diverges): $17 \rightarrow 7 \rightarrow 2 \rightarrow 16 \rightarrow 58 \rightarrow 184 \rightarrow 562 \rightarrow 1696 \rightarrow 5098 \rightarrow 15304 \ \ldots$
\item $k=17$ (The orbit diverges): $17 \rightarrow 1 \rightarrow -7 \rightarrow -11 \rightarrow -13 \rightarrow -14 \rightarrow 4 \rightarrow 58 \rightarrow 220 \rightarrow 706 \ \ldots$
\end{itemize}       
\end{example}

\begin{figure}[htp]
\centering
  \includegraphics[width=3.54in]{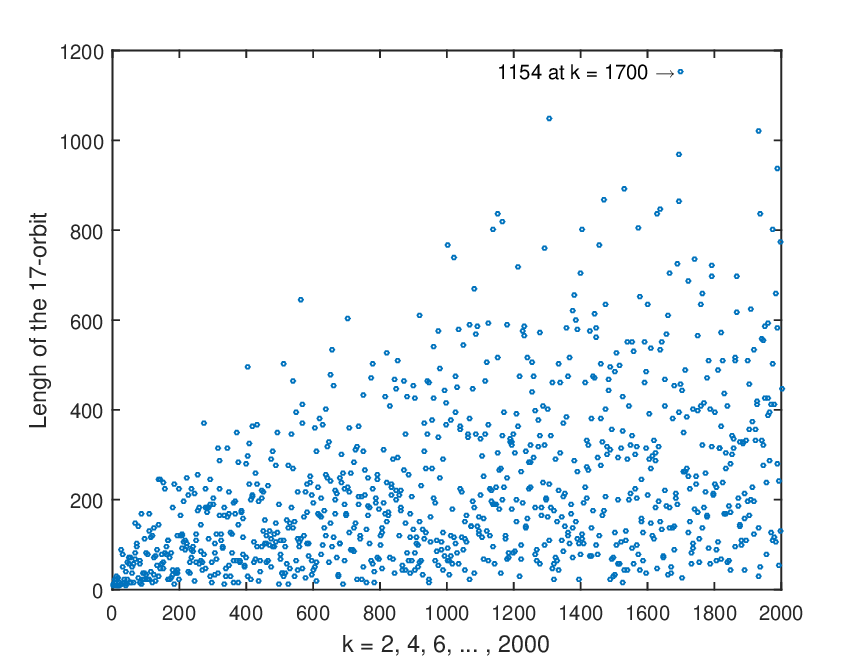}
\caption{Length of the $17$-orbit ($k \in E$).} \label{fig:1}
\end{figure}

Figure \ref{fig:1} represents the length of the $17$-orbit when $k \in E$. We can see that the length, depending of $k \in E$, is not trivial.

Example \ref{example:4} shows that if $k \in O$, the $n$-orbit diverges. This result is easy to prove. It also suggests a conjecture:
\begin{conjecture}
If $k\in E$ and $n$ is an integer, the $n$-orbit is periodic.
\end{conjecture}

The careful analysis in this section indicates that there are fundamental number properties $P$ that are equivalent on $\mathcal{Z}_{k}, \ k \in E$ and false on $\mathcal{Z}_{k}, \ k \in O$. A small modification in the definition of the product in Peano arithmetic, see Proposition \ref{propo:1}, leads to this suggestion, which should be studied with caution.

%%%%%%%%%%%%%%%%%%%%%%%%%%%%%
\section{Extension of the main idea}
\label{sec:6}
In this paper, we have seen that the usual product can be generated by the sequence $(a_n)=2,2,2,\ldots$ . With the same idea, we have considered other product mappings $\odot_{k}$ generated by the sequences $(a_n)=k,k,k,\ldots$ ($k \in \Z$). Following the same steps as in the previous sections, given an integer sequence $(a_n)=a_{1}, a_{2},\ldots, a_{n},\ldots$, we can define the product, the divisors, the quotient and the primes generated by $(a_n)$ as follows:
\begin{enumerate}
	\item \textit{Product generated by} $(a_n)$: $m$, $n \in \Z$, $n>0$,
\begin{equation*}m\odot_{a_{n}} n =(m-n+1)+(m-n+1+a_{1})+\ldots+(m-n+1+a_{1}+\ldots+a_{n-1}).\end{equation*}
	\item \textit{Divisors of an integer generated by} $(a_n)$: $d>0$,
\begin{equation*}d \mid a (\textit{arith } a_{n})\Leftrightarrow \exists b \in \Z \text{ such that } b\odot_{a_{n}} d=a.\end{equation*}
	\item \textit{Quotient generated by} $(a_n)$: $a$, $b \in \Z$, $b \neq 0$,
\begin{equation*}a \oslash_{a_{n}} b = c \Leftrightarrow c\odot_{a_{n}} b=a.\end{equation*}
	\item \textit{Prime generated by} $(a_n)$: An integer $p>1$ is called a prime (\textit{arith} $a_{n}$), or simply a $(a_{n})$-\textit{prime}, if it has only two divisors (\textit{arith} $a_{n}$). An integer greater than $1$ that is not a prime (\textit{arith} $a_{n}$) is termed composite (\textit{arith} $a_{n}$).
\end{enumerate}

Clearly, $m\odot_{a_{n}} n=(m-n+1)n+(n-1)a_{1}+(n-2)a_{2}+\ldots+1 \cdot a_{n-1}$. Thus, we can write the product generated by $(a_n)$ as follows:
\begin{equation} \label{eq:6}
m\odot_{a_{n}} n=(m-n+1)n+\sum_{i=1}^{n-1}(n-i)\cdot a_{i}.
\end{equation}

Initially, the definition applies to only positive integer $n$. However, if we can obtain a formula for $\sum_{i=1}^{n-1}(n-i)\cdot a_{i}$, we can easily extend the product for all integers, just like we did in \eqref{eq:3}.

Additionally, as in Proposition \ref{propo:3}, we can study the new quotient with the usual quotient: 
\begin{equation} \label{eq:7}
a \oslash_{a_{n}} b=\frac{1}{b} \cdot (a-\sum_{i=1}^{b-1}(b-i)a_{i})+b-1.
\end{equation}

Also, similarly to Corollary \ref{corollary:1}, it is easy to obtain the following result: 
\begin{equation} \label{eq:8}
b>0, \ b \text{ is a divisor of }a \ (arith \ a_{n}) \Leftrightarrow a \oslash_{a_{n}} b \text{ is an integer.}
\end{equation}

In the following lines, we evaluate, via an algebraic computation system, the previous results. We offer some examples of arithmetics generated by sequences and attempt to study the primes, the squares and an analogue of formula \eqref{eq:5} of Theorem \ref{theorem:3}, denoted by $\bigcup_{p}\Z \odot p$, in the simplest cases. 

\begin{example}
\label{example:5}
Obtain the expression of $\bigcup_{p}\Z \odot p$ when $(a_n)$ is an arithmetic progression whose first term is $a \in \Z$ and whose difference is $b \in \Z$.\\
\textit{Solution.} $(a_n)=a,a+b,a+2b,a+3b,\ldots$. In this case, we can obtain an explicit formula for the product $m \odot_{a_{n}} n$:
\begin{equation} \label{eq:9}
m \odot_{a_{n}} n := (m-n+1) \cdot n+\frac{ n\cdot(n-1)\cdot a}{2}+\frac{ n \cdot (n-1) \cdot (n-2) \cdot b}{6}.
\end{equation}

When the sequence $(a_n)$ is generated by a polynomial, we can always find a formula similar to \eqref{eq:9}. We can even consider the Euler-Maclaurin summation formula in \eqref{eq:6} and study the arithmetics generated by functions. However, we leave this approach for another time.

Varying $a$ and $b$, we obtain the following results for the ($a_{n}$)-primes:
\begin{enumerate}
	\item If $a \in O$ and $b\equiv 0$ (mod $3$), then the primes (\textit{arith} $a_{n}$) are:\\
	$2,4,8,16,32,64,128,256,512,1024, \ldots$ ($\{2^{s}: s \in \N \}$). 
	\item If $a \in O$ and $b\equiv 1$ (mod $3$), then the primes (\textit{arith} $a_{n}$) are:\\
	$2,6,8,18,24,32,54,72,96,128,162,216, \ldots$ ($\{2^{2s-1}\cdot 3^{t-1}: s,t \in \N \}$).
	\item If $a \in O$ and $b\equiv 2$ (mod $3$), then all $p>1$ is composite (\textit{arith} $a_{n}$).\label{example:5:3}
	\item If $a \in E$ and $b\equiv 0$ (mod $3$), then the primes (\textit{arith} $a_{n}$) are:\\ 
	$2,3,5,7,11,13,17,19,\ldots$ (usual primes).
	\item If $a \in E$ and $b\equiv 1$ (mod $3$), then the primes (\textit{arith} $a_{n}$) are:\\
	$3,9,27,81,243,729,\ldots$ ($\{3^{s}: s \in \N \}$).
	\item If $a \in E$ and $b\equiv 2$ (mod $3$), then the primes (\textit{arith} $a_{n}$) are:\\  
	$7,13,19,21,31,37,39,\ldots$ ($\{ 3^{s-1} \cdot p: s \in \N$, $p$ prime, $p\equiv 1$ (mod $6$)$\}$).
\end{enumerate}

In point \ref{example:5:3}, $a \in O$ and $b\equiv 2$ (mod $3$), all numbers have at least three divisors (\textit{arith} $a_{n}$): if $n=2^{s}\cdot h$ ($h$ odd, $s \in \{0,1,2,\ldots\}$) then, $1$, $2^{s+1}$, $6n$ are divisors of $n$ (\textit{arith} $a_{n}$). If we consider the set of integers greater than one with exactly three divisors, we obtain: $3,4,9,12,16,27,36,48,\ldots$ ($\{ 3^{i-1}\cdot 4^{j-1}>1: i,j \in \N \}$). See \seqnum{A025613}.

Now, with the previous calculation, we can compute the analogue of formula \eqref{eq:5} of Theorem \ref{theorem:3} and obtain:
\begin{enumerate}
	\item If $a \in O$ and $b\equiv 0$ (mod $3$): $\bigcup_{p}\Z \odot p = \Z \setminus \{ 0 \}$.
	\item If $a \in O$ and $b\equiv 1$ (mod $3$):\\
	$\begin{aligned} \textstyle \bigcup_{p}\Z \odot p &= \Z \setminus \{ \ldots,-14,-10,-8,-6,-2,0,2,6,8,10,14,\ldots \}=\\
	&=\Z \setminus \big \{ \{ 2^{2s-1}\cdot (2t+1): s\in \N, t \in \Z \} \cup \{  0 \} \big \}.
	\end{aligned}$
	\item If $a \in O$ and $b\equiv 2$ (mod $3$): $\bigcup_{p}\Z \odot p \ =\emptyset$.
	\item If $a \in E$ and $b\equiv 0$ (mod $3$): $\bigcup_{p}\Z \odot p \ = \ \Z \setminus \{ -1,1 \}$. 
	\item If $a \in E$ and $b\equiv 1$ (mod $3$):\\
$\begin{aligned} \textstyle \bigcup_{p}\Z \odot p &=\Z \setminus \{ \ldots,-9,-7,-4,-3,-1,0,2,5,6,8,11,14,\ldots \}= \\
&= \Z \setminus \big \{ \{ 3^{s-1}\cdot (3t+2): s \in \N, t \in \Z \} \cup \{ 0 \} \big \}. \end{aligned}$
	\item If $a \in E$ and $b\equiv 2$ (mod $3$):\\ 
	$\begin{aligned} \textstyle \bigcup_{p}\Z \odot p &=\Z \setminus \{\ldots,-6,-5,-4,-3,-2,-1,1,2,3,4,5,6,8,\ldots \}=\\
&= \{ t \cdot p: t \in \Z, \ p \text{ usual prime}, \ p\equiv 1 \ (\text{mod } 6)\}. \end{aligned}$ 
	\end{enumerate}
\end{example}

The next step could be to consider the arithmetic generated by a degree-$2$ polynomial. We could classify the results, as in Example \ref{example:5}, and finally, we could try to find a general theorem for an arithmetic generated by any polynomial. We postpone this work to another time. Notably, there are some interesting connections. For instance, if we consider the sequence $(a_n)=1, 6, 21,\ldots$, generated by the polynomial $p(x)=1+5x^{2}$, then the primes (\textit{arith} $a_{n}$) are: $2, 4, 6, 12, 18, 36, 54, \ldots$; which could also have been obtained with a chess puzzle. See \cite{elkies} and \seqnum{A068911}.

To finish the paper, motivated by Theorem \ref{theorem:1}, we study the squares of some arithmetics generated by sequences.
\begin{example} 
\label{example:6}
In the following cases, we compute $S_{a_{n}}=\{ a \odot_{a_{n}} a: a \in \N \}$.
\begin{enumerate}
	\item $(a_n)=0,1,2,3,4,\ldots \Rightarrow S_{a_{n}}=\{ 1,2,4,8,15,26, 42, 64, 93, 130 \ldots \}$. The ``cake numbers'' appear. See \seqnum{A000125}.
	\item $(a_n)= 1, 2, 4, 8, \ldots  \Rightarrow  S_{a_{n}}=\{1, 3, 7, 15, 31, 63, \ldots \}$. Mersenne numbers appear. In this concrete arithmetic, none of the Mersenne numbers is prime (\textit{arith} $a_{n}$).
	\item $(a_n)= 2, 3, 5, 7, \ldots \ (\text{usual primes}) \Rightarrow S_{a_{n}}=\{1, 4, 10, 21, 39, \ldots \}$. Convolution of natural numbers with $(1, p(1), p(2), ... )$, where $p(s)$ is the $s$-th prime. See \seqnum{A023538}.
	\item $(a_n)= 1, -1, 1, -1, \ldots \Rightarrow S_{a_{n}}=\{1,3,4,6,7,9,10,12 \ldots \}$. Numbers that are congruent to $0$ or $1$ (mod $3$). See \seqnum{A032766}. In this example is interesting to consider the set of cubes: $C_{a_{n}}=\{ (a \odot_{a_{n}} a) \odot_{a_{n}}a : \ a \in \N \}=\{ 1, 5, 7, 14, 17, 27, 31, 44, \ldots \}$. Maximum number of intersections in self-intersecting $n$-gon. See \seqnum{A105638}.
	\item $(a_n)= 0,1,0,1,0,1, \ldots  \Rightarrow  S_{a_{n}}=\{1,2,4,6,9,12,16,20,25,30, \ldots \}$. Quarter squares. See \seqnum{A002620}. If we consider the set of cubes: $C_{a_{n}}=\{ 1, 2, 7, 14, 29, 48, 79, \ldots \}$. Number of paraffins. See \seqnum{A005998}.
	\item $(a_n)= 1,-1, 0, 1,-1, 0, 0, 0, 1, -1, 0, 0, 0, 0, 0, 0, 1, -1, \ldots$ (the zeros follow the triangular number sequence) $\Rightarrow S_{a_{n}}=\{ 1, 3, 4, 5, 7, 8, 9, 10, 11, 13, 14, 15,16, 17, 18, 19, 20, 22,\\
23, \ldots \}$. This sequence coincides with the ``pancake numbers''. See \seqnum{A058986}. This coincidence deserves attention.
\end{enumerate}
\end{example}
%%%%%%%%%%%%%%%%%%%%%%%%%%%%%%%%%%%%%%%%%%%%%%%%%%%%%%%%%%
\section{Conclusion} \label{conclusion}
In this paper, we have generalized the Peano arithmetic usual product to $\odot_{k}$. This fact has allowed us to extend Furstenberg's theorem of the infinitude of primes. The study of the arithmetic generated by $\odot_{k}$ is interesting. For instance, the fundamental $k$-\textit{arithmetic} theorem makes the powers of two appear to be primes of other arithmetics. Additionally, new versions of the classical conjectures of number theory are obtained that are connected with the usual ones. This point deserves attention. Finally, the extension of the main idea invites us to consider the arithmetic generated by any integer sequence. The number of new sequences and connections that appear is enormous; hence,  more work related to this topic is necessary.

%%%%%%%%%%%%%%%%%%%%%%%%%
\subsection*{Acknowledgements}
This work was supported by King Juan Carlos University under grant C2PREDOC2020.

%%%%%%%%%%%%%%%%%%%%%%%%%%%%%%

\end{document}